\documentclass[12pt]{amsart}

\usepackage{amssymb, amsmath, amsthm}
\usepackage[margin=1in]{geometry}
\usepackage{tikz}

\DeclareMathOperator{\ex}{ex}

\newtheorem{theorem}{Theorem}[section]

\newtheorem{lemma}[theorem]{Lemma}

\newtheorem{cor}[theorem]{Corollary}

\newtheorem{conj}[theorem]{Conjecture}

\numberwithin{equation}{section}

\begin{document}

\title{The rainbow Tur\'an number of $P_5$}

\author{Anastasia Halfpap}

\maketitle

\begin{abstract}
    An edge-colored graph $F$ is {rainbow} if each edge of $F$ has a unique color. The {rainbow Tur\'an number} $\ex^*(n,F)$ of a graph $F$ is the maximum possible number of edges in a properly edge-colored $n$-vertex graph with no rainbow copy of $F$. The study of rainbow Tur\'an numbers was introduced by Keevash, Mubayi, Sudakov, and Verstra\"ete in 2007. 
    
    In this paper we focus on $\ex^*(n,P_5)$. While several recent papers have investigated rainbow Tur\'an numbers for $\ell$-edge paths $P_{\ell}$, exact results have only been obtained for $\ell < 5$, and $P_5$ represents one of the smallest cases left open in rainbow Tur\'{a}n theory. In this paper, we prove that $\ex^*(n,P_5) \leq \frac{5n}{2}$. Combined with a lower-bound construction due to Johnston and Rombach, this result shows that $\ex^*(n,P_5) = \frac{5n}{2} $ when $n$ is divisible by $16$, thereby settling the question asymptotically for all $n$. In addition, this result strengthens the conjecture that $\ex^*(n,P_{\ell}) = \frac{\ell}{2}n + O(1)$ for all $\ell \geq 3$.
    
\end{abstract}

\section{Introduction}
An \textit{edge-coloring} $c$ of a graph $G$ with edge set $E(G)$ is a function $c: E(G) \rightarrow \mathbb{N}$; for $e \in E(G)$ we call $c(e)$ the \textit{color} of $e$. We say that an edge-colored graph is \textit{properly edge-colored} if no two incident edges receive the same color, and is \textit{rainbow} if no two edges receive the same color. Often, we build (or infer) a proper edge-coloring of $G$ in steps, by beginning with a proper edge-coloring of a subgraph $H$ of $G$ and then repeatedly selecting (or deducing) the colors of edges in $E(G) \setminus E(H)$. We say that our color selections \textit{obey coloring rules} (or are \textit{legal}) if, for each edge $e \in E(G) \setminus E(H)$, the selected color $c(e)$ is not equal to $c(f)$ for any edge $f$ which is incident to $e$ and already colored. Thus, if we begin with a properly edge-colored subgraph $H$ of $G$ and then select colors for each edge in $E(G) \setminus E(H)$ in a manner which obeys coloring rules, we will return a proper edge-coloring of $G$.

Given graphs $G$ and $F$, an \textit{F-copy} in $G$ is a (not necessarily induced) subgraph of $G$ which is isomorphic to $F$; if $G$ is edge-colored, then a \textit{rainbow F-copy} in $G$ is an $F$-copy in $G$ which is rainbow under the given coloring of $G$. An edge-colored graph is {\it rainbow-$F$-free} if it contains no rainbow $F$-copy.

The \textit{rainbow Tur$\acute{a}$n number} of a fixed graph $F$ is the maximum possible number of edges in a properly edge-colored $n$-vertex rainbow-$F$-free graph $G$. We denote this maximum by $\ex^*(n,F)$, and we say that an $n$ vertex graph $G$ \textit{achieves} $\ex^*(n,F)$ if $G$ has $\ex^*(n,F)$ edges and there exists a proper edge-coloring of $G$ under which $G$ is rainbow-$F$-free. The study of rainbow Tur\'an numbers was introduced by  Keevash, Mubayi, Sudakov, and Verstra\"ete  \cite{kmsv}.

Observe that $\ex(n,F) \leq \ex^*(n,F)$, since any properly edge-colored $F$-free graph clearly contains no rainbow $F$-copy. In fact, it was proved in \cite{kmsv} that for any $F$,
\[
\ex(n,F) \leq \ex^*(n,F) \leq \ex(n,F) + o(n^2).
\]
However, for bipartite $F$, $\ex(n,F)$ and $\ex^*(n,F)$ are not asymptotic in general.
For example, in \cite{kmsv} it was shown that asymptotically $\ex^*(n,C_6)$ is a constant factor larger than $\ex(n,C_6)$. Thus, as in classical Tur\'an theory, the difficult question is to determine rainbow Tur\'an numbers for bipartite graphs. Various such problems have received attention; see, for example \cite{DLS13, OJ, jps, jr}. Here, we focus on rainbow Tur\'an numbers of paths.

We denote by $P_\ell$ the path on $\ell$ edges, i.e., $\ell+1$ vertices. The behavior of $\ex^*(n,P_\ell)$ previously has been determined asymptotically only for $\ell \leq 4$. For $\ell = 1$ and $\ell = 2$, the result $\ex^*(n, P_\ell) = \ex(n,P_\ell)$ is trivial, since any properly colored $P_1$ or $P_2$-copy is rainbow. For $\ell = 3$ and $\ell = 4$, Johnston, Palmer and Sarkar \cite{jps} determined $\ex^*(n,P_{\ell})$ asymptotically for all $n$, with exact values given certain divisibility criteria.

\begin{theorem}[Johnston-Palmer-Sarkar \cite{jps}]\label{path-edge-bounds}
If $n$ is divisible by $4$, then 
\[\ex^*(n,P_3) = \frac{3}{2}n.
\]
If $n$ is divisible by $8$, then
\[
\ex^*(n,P_4) = 2n.
\]
\end{theorem}

These results leave $P_5$ as one of the smallest graphs whose rainbow Tur\'an number has not been determined (the other notable example being $C_4$). Previously, the best known bounds on $\mathrm{ex}^*(n,P_5)$ were due to Johnston and Rombach \cite{jr} and Halfpap and Palmer \cite{HP}, respectively

\begin{theorem}[Johnston-Rombach \cite{jr}; Halfpap-Palmer \cite{HP}]\label{oldP5}

$$\frac{5}{2}n + O(1) \leq \mathrm{ex}^*(n,P_5) \leq 4n.$$

\end{theorem}

The lower bound in Theorem~\ref{oldP5} is a special case of a lower bound due to Johnston and Rombach \cite{jr}, which is the best known in general. The upper bound in Theorem~\ref{oldP5} is obtained by case analysis which does not generalize to longer paths; the best known general upper bound on $\mathrm{ex}^*(n,P_{\ell})$ is due to Ergemlidze, Gy\H{o}ri and Methuku \cite{EGyM19}.

\begin{theorem}[Johnston-Rombach \cite{jr}; Ergemlidze-Gy\H{o}ri-Methuku \cite{EGyM19}]\label{edge-bound}
For $\ell \geq 3$,
\[
\frac{\ell}{2}n + O(1)\leq \ex^*(n,P_\ell)\leq \left(\frac{9\ell+5}{7}\right)n.
\]
\end{theorem}

The lower bound on $\ex^*(n, P_{\ell})$ from \cite{jr} is achieved by taking disjoint copies of the following construction. 

\textbf{Construction 1.} Let $Q_{\ell-1}$ be the $\ell-1$ dimensional cube, i.e., the graph whose vertex set is the set of $01$-strings of length $\ell-1$ and two vertices are joined by an edge if and only if their Hamming distance is exactly $1$.

We color the edges of $Q_{\ell-1}$ by the position in which their corresponding strings differ. For each vertex $x$ of 
$Q_{\ell-1}$, let $\overline{x}$ be the {\it antipode} of $x$. That is, $\overline{x}$ is the unique vertex of Hamming distance $\ell-1$ from $x$. Now add all edges $x\overline{x}$ to this graph and color these edges with a new color $\ell$. Call these edges {\it diagonal} edges and denote the resulting edge-colored graph $D_{2^{\ell-1}}^*$. The underlying (uncolored) graph of $D_{2^{\ell-1}}^*$ is often referred to as a {\it folded cube graph}. 

\vspace{0.5 cm}

Note that the lower bound from Theorem \ref{edge-bound} is shown to be tight by Theorem \ref{path-edge-bounds} when $n = 3,4$. While these small cases provide limited data, they suggest the following.

\begin{conj} For all $\ell \geq 3$, $\ex^*(n,P_{\ell}) = \frac{\ell}{2}n + O(1)$. \end{conj}

The goal of this paper is to prove an upper bound on $\ex^*(n,P_5)$ which asymptotically matches the lower bound from Theorem \ref{edge-bound}, thereby adding further weight to the conjecture that this lower bound is asymptotically correct in all non-trivial cases.

\begin{theorem}\label{main thm}
$\frac{5n}{2} + O(1) \leq \ex^*(n,P_5) \leq \frac{5n}{2}.$
\end{theorem}

Johnston and Rombach \cite{jr} also considered a rainbow version of the generalized Tur\'an problems popularized by Alon and Shikhelman \cite{AS}. For fixed graphs $H$ and $F$, let $\ex^*(n,H,F)$ denote the maximum possible number of rainbow $H$-copies in an $n$-vertex properly edge-colored graph which is rainbow-$F$-free. (For a different formulation combining rainbow Tur\'an and generalized Tur\'an problems, see \cite{GMMP}.) We say that an $n$ vertex graph $G$ \textit{achieves} $\ex^*(n,H,F)$ if there exists a proper edge-coloring of $G$ under which $G$ is rainbow-$F$-free and contains $\ex^*(n,H,F)$ rainbow $H$-copies. Following the generalized rainbow framework in \cite{jr}, Halfpap and Palmer \cite{HP} asymptotically determined $\ex^*(n,C_{\ell},P_{\ell})$ for $\ell = 3,4,5$:

\begin{theorem}[Halfpap-Palmer \cite{HP}] \label{HPtheorem}

Let $\ell \in \{3,4,5\}$. Then, when $n$ is divisible by $2^{\ell - 1}$, we have 
$$\ex^*(n,C_{\ell},P_{\ell}) = \frac{(\ell - 1)!}{2}n;$$
moreover, when $n$ is divisible by $2^{\ell - 1}$, the graph consisting of $\frac{n}{2^{\ell-1}}$ disjoint copies of $D_{2^{\ell -1}}^*$ achieves $\ex^*(n,C_{\ell}, P_{\ell})$.

\end{theorem}

The key step in the determination of $ex^*(n, P_5, C_5)$ was the proof of the following fact, which we present here as a separate lemma.

\begin{lemma}\label{cycle-lemma}
Let $G$ be an $n$-vertex, properly edge-colored graph which is rainbow-$P_5$-free. Let $V' \subseteq V(G)$ be the set of vertices of $G$ which are contained in at least one rainbow $C_5$-copy. Then
$$\frac{\underset{v \in V'}{\sum} d(v)} {|V'|} \leq 5.$$
\end{lemma}

Consideration of $\ex^*(n, C_{\ell},P_{\ell})$ was motivated by the problem of determining $\ex^*(n,P_{\ell})$, as the extremal constructions avoiding rainbow $P_3$ and $P_4$-copies contain many rainbow $C_3$ and $C_4$-copies. Intuitively, we expect a rainbow-$P_{\ell}$-free graph with high average degree to contain many short rainbow cycles, since long rainbow walks are unavoidable. So, it is reasonable to conjecture that graphs achieving $\ex^*(n,P_{\ell})$ also achieve $\ex^*(n,C_{\ell},P_{\ell})$. 

On the other hand, Lemma \ref{cycle-lemma} suggests an approach to finding $\ex^*(n,P_5)$. Denote by $d(G)$ and $\delta(G)$ the average degree and minimum degree, respectively, of a graph $G$. Note that, to prove Theorem \ref{main thm}, it would suffice to show that if $G$ is an $n$-vertex, properly edge-colored, rainbow-$P_5$-free graph, then $d(G) \leq 5$. We may therefore consider Lemma \ref{cycle-lemma} as a partial result towards our desired theorem. The strategy of our proof is thus to focus on vertices of $G$ which do not lie in any rainbow $C_5$-copy. We will need to show that the average degree across these vertices is at most $5$.

Lemma \ref{cycle-lemma} is proved through a ``vertex pairing" argument. Consider a properly edge-colored graph $G$ which is rainbow-$P_5$-free. Partition the vertex set $V$ of $G$ as $V = V' \cup \overline{V'}$, where $V'$ is the set of vertices of $G$ which lie in some rainbow $C_5$-copy under the given edge-coloring. For $v \in V$, let $N(v)$ denote the neighborhood of $v$. We define the set of \textit{high degree} vertices in $V'$ as
$$H = \{v \in V': d(v) \geq 6\}.$$
For $v \in H$, it is shown that there exists a \textit{low degree} subset $L(v) \subset N(v) \cap V'$ such that 
$$\frac{d(v) +\sum_{u \in L(v)} d(u)}{|L(u)| + 1} \leq 5.$$
We call $\{v\} \cup L(v)$ a \textit{local pairing in} $V'$. By construction, the average degree across a local pairing in $V'$ is at most $5$. Moreover, it is shown that we can find a set of pairwise disjoint local pairings whose union contains $H$. This implies that the average degree over all of $V'$ must be at most $5$. 

We shall use a similar approach to estimate the average degree of vertices in $\overline{V'}$. A \textit{local pairing in} $\overline{V}'$ is defined analogously, as a set $\{v\} \cup L(v)$, where $d(v) \geq 6$, $L(v) \subseteq N(v)$, $\{v\} \cup L(v) \subseteq \overline{V}'$, and 
$$\frac{d(v) +\sum_{u \in L(v)} d(u)}{|L(u)| + 1} \leq 5.$$
The crux of the argument is to find a set of pairwise disjoint local pairings in $\overline{V'}$ which contain all high degree vertices of $\overline{V'}$.

The paper is organized as follows. In Section 2, we prove a number of structural lemmas relating to vertices which do not lie in any rainbow $C_5$-copy. In Section 3, we employ these lemmas to prove our main result.

\section{Structural Lemmas}

Fix $n \geq 1$ and let $G_0$ be a properly edge-colored, $n$-vertex graph achieving $\ex^*(n, P_5)$. We may assume that $e(G_0) > \frac{5n}{2}$, since otherwise we will have nothing to prove. Thus, $d(G_0) > 5$. We will modify $G_0$ by repeatedly pruning vertices of degree $1$ or $2$. By a standard argument, pruning in this way yields a subgraph with minimum degree at least $3$, whose average degree is at least that of $G_0$. From this subgraph, we will moreover delete any components which have average degree at most $5$. Denote by $G$ the resulting subgraph of $G_0$. Observe that $$d(G) \geq d(G_0) > 5.$$ 

Throughout, we work with this graph $G$. We will ultimately achieve a contradiction by arguing that in fact, $d(G) \leq 5$, thus implying our main result. 

Let $V$ be the vertex set of $G$. We will form the partition $V = V' \cup \overline{V'}$, where $V'$ is the set of vertices in $G$ which appear in at least one rainbow $C_5$-copy, and $\overline{V'}$ is the set of vertices which do not appear in a rainbow $C_5$-copy. As noted above, since Lemma \ref{cycle-lemma} tells us that the average degree over $V'$ is at most $5$, it will suffice to show that the average degree over $\overline{V'}$ is at most $5$. 

We will follow a similar approach as in the proof of Lemma \ref{cycle-lemma}. Rather than prove that $\overline{V'}$ contains no vertex of degree greater than $5$ (which, even if true, seems too difficult to directly prove by case analysis), we shall estimate the average degree in $\overline{V}$. We do this by finding local pairings in $\overline{V'}$; once we have found these, we moreover argue that their existence implies that globally, $\overline{V'}$ in fact has average degree at most 5.

The first step is to show that, if $v \in \overline{V'}$ has $d(v) \geq 6$, then the neighbors of $v$ are also in $\overline{V'}$. Since we will eventually pair $v$ to its neighbors, this will ensure that we are pairing $v$ only to other vertices in $\overline{V'}$. 

\begin{lemma}\label{neighbor-lemma}
If $v \in \overline{V'}$ has $d(v) \geq 6$, then $N(v) \subset \overline{V'}$.

\end{lemma}

\begin{proof}

Suppose for a contradiction that $v$ has a neighbor, $u$, which is in $V'$. So $u$ lies on a rainbow $C_5$-copy, which we shall call $C$. We may assume that the edges of $C$ are colored from $\{1,2,3,4,5\}$, as pictured in  Figure~\ref{fig1}.

\begin{figure}[h]
\centering
\begin{tikzpicture}
\filldraw (-1,0) circle(0.05 cm);
\filldraw (1,0) circle(0.05 cm);
\filldraw (-1,1) circle(0.05 cm);
\filldraw (1,1) circle(0.05 cm);
\filldraw (0,2) circle(0.05 cm);
\draw (-1,0) -- (1,0) node[pos=0.5, below] {4};
\draw (-1,0) -- (-1,1) node[pos=0.5, left] {5};
\draw (1,0) -- (1,1) node[pos=0.5, right] {3};
\draw (-1,1) -- (0,2) node[pos=0.6, left] {1};
\draw (1,1) -- (0,2) node[pos=0.6, right] {2};
%\draw (1,0) node[right] {$v_1$};
%\draw (-1,0) node[left] {$v_2$};
%\draw (-1,1) node[left] {$v_3$};
\draw (0,2) node[above left] {$u$};
%\draw (1,1) node[right] {$v_5$};
\filldraw (0,3.5) circle(0.05 cm);
\draw (0,2) -- (0,3.5);
\draw (0,3.5) node[left]{$v$};

\end{tikzpicture}
\caption{}\label{fig1}
\end{figure}

Observe that there is no edge with precisely one endpoint incident to $C$ that is colored with a color not in $\{ 1,2,3,4,5 \}$, as this immediately creates a rainbow $P_5$-copy. Observe also that since $d(v) \geq 6$, $v$ is incident to an edge which is not colored from $\{ 1,2,3,4,5\}$. We call the color on this edge $6$, and conclude that the other endpoint, say $w$, of this edge is not depicted in Figure \ref{fig1}. So the situation is as in Figure~\ref{fig2}. We will also add, in Figure~\ref{fig2}, labels to the remaining vertices on $C$.

\begin{figure}[h]
\centering
\begin{tikzpicture}
\filldraw (-1,0) circle(0.05 cm);
\filldraw (1,0) circle(0.05 cm);
\filldraw (-1,1) circle(0.05 cm);
\filldraw (1,1) circle(0.05 cm);
\filldraw (0,2) circle(0.05 cm);
\draw (-1,0) -- (1,0) node[pos=0.5, below] {4};
\draw (-1,0) -- (-1,1) node[pos=0.5, left] {5};
\draw (1,0) -- (1,1) node[pos=0.5, right] {3};
\draw (-1,1) -- (0,2) node[pos=0.6, left] {1};
\draw (1,1) -- (0,2) node[pos=0.6, right] {2};
\draw (1,0) node[right] {$u_2$};
\draw (-1,0) node[left] {$u_3$};
\draw (-1,1) node[left] {$u_4$};
\draw (0,2) node[above left] {$u$};
\draw (1,1) node[right] {$u_1$};
\filldraw (0,3.5) circle(0.05 cm);
\draw (0,2) -- (0,3.5);
\draw (0,3.5) node[left]{$v$};
\filldraw (0,5) circle(0.05 cm);
\draw (0,3.5) -- (0,5) node[pos=0.5, right]{6};
\draw (0,5) node[left]{$w$};

\end{tikzpicture}
\caption{}\label{fig2}
\end{figure}

Note now that $c(uv)$ must equal $4$, or else either of $u_3 u_2 u_1 u v w$ or $u_2 u_3 u_4 u v w$ is a rainbow $P_5$-copy. 

Now, we shall examine $w$. Recall that $\delta(G) \geq 3$. We will aim to arrive at a contradiction by showing that $d(w) \leq 2$. 

First, we shall observe that $w$ can be adjacent to no vertex of $C$. Indeed, if $wu$ is an edge, then $c(wu) \in \{1,2,3,4,5\}$ to avoid an immediate rainbow $P_5$-copy. Since the coloring must be proper, this means that $c(uw) \in \{ 3,5 \}$. But either choice produces a rainbow $P_5$-copy (either $v w u u_4 u_3 u_2$ or $v w u u_1 u_2 u_3$). So $w u$ is not an edge.

Next, observe that if $w u_1$ is an edge, then $c(w u_1) \neq 5$, since then $u_4 u v w u_1 u_2 $ is a rainbow $P_5$-copy. But this implies that either $v w u_1 u u_4 u_3$ or $v w u_1 u_2 u_3 u_4$ is a rainbow $P_5$-copy.  So $w u_1$ is not an edge. Analogously, $w u_4$ is not an edge.

Finally, if $w u_2$ is an edge, then it must be colored 1, since otherwise either $v w u_2 u_3 u_4 u$ or $v w u_2 u_1 u u_4$ is a rainbow $P_5$-copy. But if $c(w u_2) = 1$, then $w u_2 u_1 u v w$ is a rainbow $C_5$-copy containing $v$, a contradiction, since we assume $v \in \overline{V'}$. By an analogous argument, $w u_3$ is not an edge.

Thus, $w$ is not adjacent to any vertex on $C$. Moreover, if $x$ is a vertex not yet considered such that $w x$ is an edge, then observe that $c(wx)$ must equal $4$, or else either $x w v u u_1 u_2$ or $x w v u u_4 u_3$ is a rainbow $P_5$-copy. Thus, to avoid a rainbow $P_5$-copy, we must have $d(w) \leq 2$, contradicting the minimum degree condition on $G$. 

We conclude that $v$ is not adjacent to any vertex in $V'$.
\end{proof}

We will also make use of the following lemma, which further restricts the subgraphs in which high-degree vertices of $\overline{V'}$ may appear. In a $P_{\ell}$-copy $P$, the \textit{endpoints} of $P$ are the two vertices whose degree in $P$ is 1. 

\begin{lemma}\label{length-4}
Suppose $v \in \overline{V'}$ has $d(v) \geq 6$. Then $v$ is not the endpoint of a rainbow $P_4$-copy.
\end{lemma}

\begin{proof}

Suppose that $v \in \overline{V'}$ has $d(v) \geq 6$ and is the endpoint of a rainbow $P_4$-copy, say $P$. Our ultimate strategy is to obtain a contradiction by arguing that, to avoid a rainbow $P_5$-copy, $G$ must contain a vertex of degree strictly less than $3$. In order to identify this low-degree vertex, we will first need to perform some case analysis to understand the structure of $G$ near $v$. We begin by drawing $P$ in Figure~\ref{fig3}, also labeling the other vertices and indicating edge colors.

\begin{figure}[h]
\centering
\begin{tikzpicture}
\filldraw (0,0) circle(0.05 cm) node[below]{$w$};
\filldraw (2,0) circle(0.05 cm) node[below]{$z$};
\filldraw (4,0) circle(0.05 cm) node[below]{$y$};
\filldraw (6,0) circle(0.05 cm) node[below]{$x$};
\filldraw (8,0) circle(0.05 cm) node[below]{$v$};
\draw (0,0) -- (8,0);
\draw (1, 0) node[above]{$4$};
\draw (3, 0) node[above]{$3$};
\draw (5, 0) node[above]{$2$};
\draw (7, 0) node[above]{1};

\end{tikzpicture}
\caption{}\label{fig3}
\end{figure}

Now, $d(v) \geq 6$, so $v$ is incident to at least two edges which are not colored from $\{1,2,3,4\}$. Clearly, both endpoints of these edges must be on $P$ to avoid a rainbow $P_5$-copy. Also, since we assume that $v$ is not in a rainbow $C_5$-copy, $vw$ cannot be such an edge. So, in fact, the two edges which are incident to $v$ and not colored from $\{1,2,3,4\}$ must be $vy$ and $vz$. Without loss of generality, $c(vy) = 5$  and $c(vz) = 6$. Every other edge incident to $v$ must be colored from $\{1,2,3,4\}$; in particular, to achieve $d(v) \geq 6$, $v$ must be incident to an edge colored $4$. The other endpoint of this edge cannot be on $P$, since $v$ is already adjacent to $x,y,z$ via edges of different colors, and $w$ is already incident to another edge of color $4$. So $v$ is incident to a vertex not yet drawn, say $u$, with $c(vu) = 4$. We update our drawing in Figure~\ref{fig4}.

\begin{figure}[h]
\centering
\begin{tikzpicture}
\filldraw (0,0) circle(0.05 cm) node[below]{$w$};
\filldraw (2,0) circle(0.05 cm) node[below]{$z$};
\filldraw (4,0) circle(0.05 cm) node[below]{$y$};
\filldraw (6,0) circle(0.05 cm) node[below]{$x$};
\filldraw (8,0) circle(0.05 cm) node[below]{$v$};
\filldraw (10,0) circle(0.05 cm) node[below]{$u$};
\draw (0,0) -- (10,0);
\draw (1, 0) node[above]{$4$};
\draw (3, 0) node[above]{$3$};
\draw (5, 0) node[above]{$2$};
\draw (7, 0) node[above]{1};
\draw (9,0) node[above]{4};
\draw (8,0) arc (0:180:3);
\draw (8,0) arc (0:180:2);
\draw (5,1.5) node{5};
\draw (3,2) node{6};
\end{tikzpicture}
\caption{}\label{fig4}
\end{figure}

We will now investigate potential neighbors of $w$. Firstly, we claim that $w$ can be adjacent only to vertices already depicted in Figure \ref{fig3}. Indeed, suppose $w'$ is a vertex not yet represented, and that $ww'$ is an edge. Obeying coloring rules, we must have $c(ww') \in \{ 1,2,3\}$, or else $v x y z w w'$ is a rainbow $P_5$-copy. However, whichever of these colors we chose, either $x y v z w w'$ or $x v y z w w'$ is rainbow. Thus, $N(w) \subseteq \{x,y,z,v,u\}$. Since none of $v x y z w w'$, $x y v z w w'$, $x v y z w w'$ use the vertex $u$, these observations also imply that $wu$ is not an edge.

We also claim that $wx$ is not an edge. Indeed, suppose edge $wx$ is present. Then $v y x w z v$ is a $C_5$-copy containing $v$, so must not be rainbow. Thus, $c(wx)$ must be in $\{5,6\}$ to obey coloring rules and avoid a rainbow $C_5$-copy. But then either $x v y z w x$ is a rainbow $C_5$-copy containing $v$, or $u v z y x w$ is a rainbow $P_5$-copy. We concude that $wx$ is not an edge.

Thus, the only possible neighbors of $w$ are $z,y,$ and $v$. Since $\delta(G) \geq 3$, $w$ must be adjacent to all three of these vertices to avoid a contradiction. We can check that to obey coloring rules and avoid rainbow-$C_5$ copies containing $v$, we must have $c(wv) \in \{2,3\}$ and $c(wy) \in \{1,6\}$. We will set $c(wy) = a$ and $c(wv) = b$.

Note also that we have now accounted for five neighbors of $v$; another must exist, and must be a vertex not depicted in Figure \ref{fig4}. We shall call this neighbor $s$. To ensure that $s v x y z w$ is not a rainbow $P_5$-copy, we must have $c(sv) \in \{1, 2, 3, 4 \}$. By coloring rules, $c(sv) \in \{2,3\}$. We shall set $c(sv) = c$. Note also that $b \neq c$, so if one of $b,c$ is determined, then the other is also.

We shall reflect our progress in Figure~\ref{fig5}.

\begin{figure}[h]
\centering
\begin{tikzpicture}
\filldraw (0,0) circle(0.05 cm) node[below left]{$w$};
\filldraw (2,0) circle(0.05 cm) node[below]{$z$};
\filldraw (4,0) circle(0.05 cm) node[below right]{$y$};
\filldraw (6,0) circle(0.05 cm) node[below]{$x$};
\filldraw (8,0) circle(0.05 cm) node[below right]{$v$};
\filldraw (10,0) circle(0.05 cm) node[below]{$u$};
\draw (0,0) -- (10,0);
\draw (1, 0) node[above]{$4$};
\draw (3, 0) node[above]{$3$};
\draw (5, 0) node[above]{$2$};
\draw (7, 0) node[above]{1};
\draw (9,0) node[above]{4};
\draw (8,0) arc (0:180:3);
\draw (8,0) arc (0:180:2);
\draw (5,1.5) node{5};
\draw (3,2) node{6};
\draw (3,-1.5) node{$a$};
\draw (5, -3.5) node{$b$};

\filldraw (10,1.5) circle (0.05 cm) node[above]{$s$}; 
\draw (8,0) -- (10,1.5) node[pos=0.5, above]{$c$};

\draw (0,0) arc (180:360:2);
\draw (0,0) arc (180:360:4);

\end{tikzpicture}
\caption{}\label{fig5}
\end{figure}

We shall next examine $x$. We begin by observing that $x$ is not a neighbor of $z$ or $u$. Recall that we have already shown that $x$ is not a neighbor or $w$. 

Suppose $xz$ is an edge. By coloring rules, we have $c(xz) \in \{5,7\}$. Therefore, $u v x z y w$ is a rainbow-$P_5$ copy unless $a = 1$. Given $a = 1$, we have that $s v z x y w$ is a rainbow $P_5$-copy unless $c = 2$. This implies $b = 3$. But now $u v w y x z$ is a rainbow-$P_5$ copy. We conclude that $xz$ is not an edge.

Next, suppose $xu$ is an edge. We must have $c(xu) \in \{2,4,5,6\}$, else $w z v y x u$ is a rainbow $P_5$-copy, and $c(xu) \in \{1,3,4,5\}$, else $uxvyzw$ is a rainbow $P_5$-copy. Thus, $c(xu) \in \{4,5\}$. By coloring rules, $c(xu) \neq 4$. Moreover, if $c(xu) = 5$, then $v u x y z v$ is a rainbow $C_5$-copy containing $v$. Thus, $xu$ also is not an edge.

Now, if $d(x) \geq 3$, then either $xs$ is an edge, or $x$ has a neighbor, say $t$, which is not depicted in Figure~\ref{fig5}. We examine the cases separately; in each, we show that to avoid a rainbow $P_5$-copy, $G$ must contain a vertex of degree at most $2$.

\textbf{Case 1:} $xs$ is an edge.

Observe that $c(xs) = 4$, else one of $s x v y z w$, $s x y v z w$, $s x y z v u$ is a rainbow $P_5$-copy. Consider potential neighbors of $s$. Observe that $su$ is not an edge, since if so, one of $u s x y v z$, $u s x v y z$, $u s x y z v$ is a rainbow $P_5$-copy. Analogously, $s$ can be adjacent to no vertex which is not depicted in Figure~\ref{fig5}. We have also argued previously that $w$ is only adjacent to $z, y,$ and $v$, so $sw$ is not an edge. 

We next consider $z$. If $sz$ is an edge, observe that $c(sz) \in \{1,2\}$, since $u v x y z s$ is a rainbow $P_5$-copy under any other legal color assignment. Moreover, if $c(sz) = 2$, then $c = 3$, and $zsvywz$ is a rainbow $C_5$-copy containing $v$. So $c(sz) = 1$. Now, consider the paths $s z w v y x$ and $x y v s z w$. The colors on these paths are, respectively, $1, 4, b, 5, 2$ and $2, 5, c, 1, 4$. Recall that $b \in \{2,3\}$ and $c \in \{2,3\}$; moreover, $b \neq c$. The first path is rainbow unless $b = 2$, and the second path is rainbow unless $c = 2$; since both cannot simultaneously be true, we conclude that one path is a rainbow $P_5$-copy. Thus, $sz$ is not an edge. 

Now, to avoid a contradiction to the minimum degree condition on $G$, $sy$ must be an edge. We first claim that $c(sy) = 1$. Observe that both $x v s y z w$ and $z y s x v w$ are $P_5$-copies, respectively colored $1, c, c(sy), 3, 4$ and $3, c(sy), 4, 1, b$. Recall that either $b = 2$ or $c = 2$, so the colors on one of these two paths are (not in order) $1,2,3,4,c(sy)$. By coloring rules, $c(sy)$ is not in $\{2,3,4\}$ (since $c(yx) = 2, c(yz) = 3$, and $c(sx) = 4$), so we must have $c(sy) = 1$. Since $c(sy) \neq a$ by coloring rules, and $a \in \{ 1,6 \}$, this means $a = 6$. Now, $z y w v x s$ is a $P_5$-copy colored $3,6,b,1,4$, so we must have $b = 3$ to avoid a rainbow $P_5$-copy, forcing $c = 2$. 

Thus, if $xs$ is an edge, then $sy$ must also be an edge, and we can fix the colors of all edges indicated in Figure~\ref{fig5}. We reflect this state of affairs in Figure~\ref{fig6}.

\begin{figure}
\centering
\begin{tikzpicture}
\filldraw (0,0) circle(0.05 cm) node[below left]{$w$};
\filldraw (2,0) circle(0.05 cm) node[below]{$z$};
\filldraw (4,0) circle(0.05 cm) node[below right]{$y$};
\filldraw (6,0) circle(0.05 cm) node[below]{$x$};
\filldraw (8,0) circle(0.05 cm) node[below right]{$v$};
\filldraw (10,0) circle(0.05 cm) node[below]{$u$};
\draw (0,0) -- (10,0);
\draw (1, 0) node[above]{$4$};
\draw (3, 0) node[above]{$3$};
\draw (5, 0) node[above]{$2$};
\draw (7, 0) node[above]{1};
\draw (9,0) node[above]{4};
\draw (8,0) arc (0:180:3);
\draw (8,0) arc (0:180:2);
\draw (5,1.5) node{5};
\draw (3,2) node{6};
\draw (3,-1.5) node{$6$};
\draw (5, -3.5) node{$3$};

\filldraw (10,1.5) circle (0.05 cm) node[above right]{$s$}; 
\draw (8,0) -- (10,1.5) node[pos=0.5, above]{$2$};

\draw (0,0) arc (180:360:2);
\draw (0,0) arc (180:360:4);

\draw (6,0) arc (180:43:2.3);

\draw (4,0) arc (-180:30:3.2);

\draw (8,2.85) node[below right]{4};

\draw (10,-1.5)  node[below right]{1};

\end{tikzpicture}
\caption{}\label{fig6}
\end{figure}

Given the configuration in Figure~\ref{fig6}, we shall observe that $u$ cannot satisfy the minimum degree condition. We have seen already that $uw$ and $ux$ are not edges, and have seen that, given $xs$ is an edge, $su$ is not an edge. Observe that $u$ is adjacent to no vertex not yet drawn (say $t$), since if so, one of $t u v x y z$, $t u v s y w$, $t u v w y s$, $t u v w y x$ is a rainbow $P_5$-copy. Observe also that $u$ cannot be adjacent to $y$; if so, by coloring rules, $c(uy)$ must be a color not yet used, say $7$, and then $u y s v z w$ is a rainbow $P_5$-copy. Thus, $d(u) \leq 2$.

%We conclude that $xs$ is not an edge. This finishes the work to establish that among the vertices already drawn ($v,x,y,z,w,u,s$), $x$ is adjacent only to $v$ and $y$. 

\textbf{Case 2:} $xs$ is not an edge; thus, $x$ has a neighbor $t$ which is not depicted in Figure~\ref{fig5}.

Observe that $c(xt) = 4$, else one of $t x v y z w$, $t x y v z w$, $t x y z v u$ is a rainbow $P_5$-copy. Therefore, $G$ contains the subgraph drawn in Figure~\ref{fig7}.

\begin{figure}[h]
\centering
\begin{tikzpicture}
\filldraw (0,0) circle(0.05 cm) node[below left]{$w$};
\filldraw (2,0) circle(0.05 cm) node[below]{$z$};
\filldraw (4,0) circle(0.05 cm) node[below right]{$y$};
\filldraw (6,0) circle(0.05 cm) node[below right]{$x$};
\filldraw (8,0) circle(0.05 cm) node[below right]{$v$};
\filldraw (10,0) circle(0.05 cm) node[below]{$u$};
\draw (0,0) -- (10,0);
\draw (1, 0) node[above]{$4$};
\draw (3, 0) node[above]{$3$};
\draw (5, 0) node[above]{$2$};
\draw (7, 0) node[above]{1};
\draw (9,0) node[above]{4};
\draw (8,0) arc (0:180:3);
\draw (8,0) arc (0:180:2);
\draw (5,1.5) node{5};
\draw (3,2) node{6};
\draw (3,-1.5) node{$a$};
\draw (5, -3.5) node{$b$};

\filldraw (10,1.5) circle (0.05 cm) node[above]{$s$}; 
\draw (8,0) -- (10,1.5) node[pos=0.5, above]{$c$};

\draw (0,0) arc (180:360:2);
\draw (0,0) arc (180:360:4);

\filldraw (6, - 2) circle(0.05 cm) node[right]{$t$};
\draw (6,0) -- (6,-2) node[pos= 0.5, left]{4};

\end{tikzpicture}
\caption{}\label{fig7}
\end{figure}

We argue that $d(t) \leq 2$. 

Observe that $t$ is not adjacent to any vertex which is not yet drawn (say $r$), else one of $r t x y v z$, $r t x v y z$, $r t x v z y$ is a rainbow-$P_5$ copy. Analogously, $t$ is not adjacent to $s$ or $u$. We have already seen that $w$ is not adjacent to $t$ (or indeed, to any vertex not in $\{z,y,v\}$). $v$ already has incident edges of every color from $\{1,2,3,4,5,6\}$ (since $b,c \in \{2,3\}$), so $tv$ cannot be an edge, as it would receive a new color, say $7$, making $t v x y z w$ a rainbow-$P_5$ copy. So, the only vertices to which $t$ can be adjacent (aside from $x$) are $y$ and $z$. 

Suppose $ty$ is an edge. By coloring rules, $c(ty)$ is not in $\{a, 2, 3, 4\}$. Observe that $s v w y t x$ is now a $P_5$-copy, with colors $c, b, a, c(ty), 4$. We know that $b,c$ are in $\{2,3\}$ and are not equal, and that $a \in \{1,6\}$ is not equal to $c(ty)$. So, $c,b,a,c(ty),$ and $4$ must all be distinct colors, and thus $ty$ is not an edge, as its presence yields a rainbow $P_5$-copy. Thus, $d(t) \leq 2$, since its only neighbors are $x$ and possibly $z$.

Thus, to avoid a rainbow $P_5$-copy in $G$, either $d(x) \leq 2$ or $x$ has a neighbor of degree at most 2. In any case, we achieve a contradiction to the minimum degree hypothesis on $G$. We conclude that, if $v \in \overline{V'}$ has $d(v) \geq 6$, then $v$ is not the endpoint of a rainbow $P_4$-copy. \end{proof}

From Lemma \ref{length-4}, we can quickly derive the following corollary, which will also be of use in our proof of the main result. The \textit{distance} between vertices $x,y$ in a graph $G$ is the smallest $\ell$ such that $G$ contains a $P_{\ell}$-copy with endpoints $x,y$. (If no such $\ell$ exists, we say that $x,y$ are at infinite distance.)

\begin{cor}\label{distance-2}
Suppose $v \in \overline{V'}$ has $d(v) \geq 6$ and $u$ is a vertex at distance $2$ from $v$ with $d(u) \geq 6$. Then $u$ is also in $\overline{V'}$.
\end{cor}

\begin{proof}
Suppose for a contradiction that $u$ is contained in a rainbow $C_5$-copy, $C$. By assumption, there exists a $P_2$-copy connecting $v$ and $u$, which is necessarily rainbow. We shall label the edges of $P$ with colors $1,2$, and the edges of $C$ with colors $a,b,c,d,e$. Note that, since $d(u) \geq 6$, $u$ must be incident to an edge of a color $f$ which is not contained in $\{a,b,c,d,e\}$. The other endpoint of this edge must be on $C$, or a rainbow $P_5$-copy is immediately created. We depict this in Figure~\ref{fig8}, adding labels to previously unnamed vertices for convenience.

\begin{figure}[h]
\centering
\begin{tikzpicture}
\filldraw (0,0) circle(0.05 cm) node[below]{$v$};
\filldraw (2,0) circle(0.05 cm) node[below]{$w$};
\filldraw (4,0) circle(0.05 cm) node[below]{$u$};
\filldraw (6,2) circle(0.05 cm) node[above]{$u_1$};
\filldraw (6,-2) circle(0.05 cm) node[below]{$u_4$};
\filldraw (8,2) circle(0.05 cm) node[right]{$u_2$};
\filldraw (8,-2) circle(0.05 cm) node[right]{$u_3$};

\draw (0,0) -- (2,0) node[pos=0.5, above]{1};
\draw (2,0) -- (4,0) node[pos=0.5, above]{2};
\draw (4,0) -- (6,2) node[pos=0.5, above left]{$a$};
\draw (6,2) -- (8,2) node[pos=0.5, above]{$b$};
\draw (8,2) -- (8,-2) node[pos=0.5, right]{$c$};
\draw (8,-2) -- (6,-2) node[pos=0.5, below]{$d$};
\draw (6,-2) -- (4,0) node[pos=0.5, below left]{$e$};
\draw (4,0) -- (8,2) node[pos=0.5, below right]{$f$};

\end{tikzpicture}
\caption{}\label{fig8}
\end{figure}

Now, since $C$ is rainbow and $f$ is distinct from $a,b,c,d,e$, at most two of $a,b,c,d,e,f$ are in $\{1,2\}$. Thus, one of $v w u u_1 u_2$, $v w u u_2 u_3$, $v w u u_4 u_3$ is a rainbow $P_4$-copy ending at $v$, a contradiction by Lemma \ref{length-4}
\end{proof}

Finally, we will need the following result. Although the next lemma holds for any $v \in V$, we will apply it in particular to vertices of high degree in $\overline{V'}$, in order to more easily build rainbow $P_4$-copies. 

\begin{lemma}\label{P3 lemma}
Let $v \in V$. Then $v$ is the endpoint of a rainbow $P_3$-copy.
\end{lemma}

\begin{proof}

Suppose $v \in V$ is not the endpoint of a rainbow $P_3$-copy. We claim that the following holds. If $u$ is a neighbor of $v$, then $d(u) = 3$ and $N(u) \subseteq N(v) \cup \{v\}$. If we can establish this claim, then the lemma is proved, as follows. Given that for any $u \in N(v)$, we have $N(u) \subseteq N(v) \cup \{v\}$, we observe that $\{v\} \cup N(v)$ must induce a component of $G$. Moreover, since every vertex in $N(v)$ has degree 3, the average degree in this component is 
$$\frac{d(v) + 3d(v)}{d(v) + 1} < 5,$$
a contradiction, as we suppose that every component of $G$ has average degree greater than $5$. Thus, we will have that every $v \in V$ is the endpoint of a rainbow $P_3$-copy.

To establish the claim, suppose $v \in V$ is not the endpoint of a rainbow $P_3$-copy, and let $u$ be a neighbor of $v$. Without loss of generality, $c(uv) = 1$. Since $\delta(G) \geq 3$, $u$ has another neighbor, say $x$, and $c(ux)$ cannot equal $c(uv)$. Say $c(ux) = 2$. Now, $x$ has at least two more neighbors, so must have at least one neighbor not equal to $v$, say $y$. Since we assume $v$ is not the endpoint of a rainbow $P_3$-copy, we must have $c(xy) = 1$. Observe that $x$ can have no other neighbor except $v$ without creating a rainbow $P_3$-copy ending in $v$, so $xv$ must be an edge. We must have $c(xv) = 3$, since $x$ is already incident to edges of colors $1$ and $2$.

We consider the neighbors of $u$; there must be at least one more. Either $uy$ is an edge, or $u$ is adjacent to some vertex $w$ not already considered. We wish to show that there is no such edge $uw$. If there is, then $c(uw) = 3$ to avoid a rainbow-$P_3$ ending at $v$. Thus, $u$ is adjacent to only one new vertex $w$. Now, $uy$ cannot be an edge, since then $c(uy)$ would be a new color, say $4$, which would create a rainbow-$P_3$ copy ending at $v$. So $y$ is adjacent to two more vertices, neither of which are $u$. If $y$ is adjacent to a vertex $z$ not already considered, then we must have $c(yz) = 3$ to avoid a rainbow $P_3$-copy ending at $v$. $y$ cannot be adjacent to $w$, as $c(yw)$ would be $2$ or $4$, creating a rainbow $P_3$-copy ending at $v$. So to achieve degree $3$, $y$ must be adjacent to $v$ and a new vertex, $z$. We have $c(yz) = 3$, and must have $c(yv) = 2$ to avoid a rainbow $P_3$-copy ending at $v$.

Now, consider $w$. We have already noted that $wy$ is not an edge; nor are $wv$ or $wx$, since these would necessarily receive a new color and thus create a rainbow-$P_3$ copy ending at $v$. Any coloring of $wz$ which is permitted by coloring rules also creates a rainbow $P_3$-copy ending at $v$, so neither is $wz$ an edge. Finally, if $w$ is adjacent to a new vertex, say $s$, then $c(ws) = 1$ to avoid a rainbow $P_3$-copy ending at $v$, so $w$ has at most one neighbor not already considered. But this implies that $d(w) \leq 2$, a contradiction. We conclude that $uw$ is not an edge. 

Thus, to achieve degree $3$, $u$ is adjacent to $y$, and $c(uy) = 3$ is forced. Given this edge, $y$ must be adjacent to $v$ to achieve degree $3$, with $c(yv) = 2$. Now, $v,u,x,y$ form a properly colored $K_4$. It is clear that if any of $u,x,y$ have another neighbor, the incident edge will create a rainbow $P_3$-copy ending at $v$, so in particular, $N(u) = \{v, x, y\} \subset N(v) \cup \{v\}$, and $d(u) = 3$. Since $u$ was chosen from $N(v)$ arbitrarily, we are done.  
\end{proof}

\section{Main Result}

We are now ready to prove our main result. As in Section 2, we may assume that we work in a properly edge-colored, rainbow-$P_5$-free graph $G$ with $\delta(G) \geq 3$ and such that every component of $G$ has average degree greater than $5$. The vertex set of $G$ is again partitioned as $V' \cup \overline{V'}$, where $V'$ is the set of vertices which lie in some rainbow $C_5$-copy in $G$. 

\begin{theorem} 
$\mathrm{ex}^*(n, P_5) \leq \frac{5n}{2}$. 
\end{theorem}

\begin{proof}

Our goal is to show, for a contradiction, that $d(G) \leq 5$. By Lemma \ref{cycle-lemma}, we know that the average degree in $V'$ is at most $5$, so we will be done if we can also show that $\overline{V'}$ has average degree at most 5. Define $H(\overline{V'}) := \{v \in \overline{V'}: d(v) > 5\}$. For each $v \in H(\overline{V'})$, we aim to find a set $L(v) \subseteq N(v)$ such that the following hold:

\begin{enumerate}
    \item $\{ v\} \cup L(v)$ is a local pairing in $\overline{V'}$
    \item For any $u \in L(v)$, we have $d(u) \leq 5$ and $N(u) \cap H(\overline{V'}) = \{v\}.$
    
\end{enumerate}

Suppose that for every $v \in H(\overline{V'})$, we can find such a set $L(v)$. Then it immediately follows that $\overline{V'}$ has average degree at most $5$. Indeed, by the definition of a local pairing in $\overline{V'}$, condition (1) gives that each $\{v\} \cup L(v)$ is a subset of $\overline{V'}$ with average degree at most 5. By condition (2), each vertex in $\overline{V'}$ appears in at most one local pairing $\{v\} \cup L(v)$. Thus, setting $S$ to be the set of vertices in $\overline{V'}$ which are not contained in any of the selected local pairings in $\overline{V'}$, we have
$$\sum_{v \in \overline{V'}} d(v) = \sum_{v \in H(\overline{V'})} \left( d(v) + \sum_{u \in L(v)} d(u) \right) + \sum_{v \in S} d(v) \leq \sum_{v \in H(\overline{V'})} 5(|L(v)|+1) + 5|S| = 5|\overline{V'}|.$$

It thus only remains to show that such a set $L(v)$ can be found for every vertex $v \in H(\overline{V'})$. By Lemma $\ref{neighbor-lemma}$, we have that $\{v\} \cup N(v) \subseteq \overline{V'}$ for any $v \in H(\overline{V'})$, so any potential local pairing will indeed be contained in $\overline{V'}$. Fix $v \in H(\overline{V'})$ and consider $N(v)$. We distinguish two cases. Throughout both, recall that by Lemma \ref{length-4}, we may assume that $v$ is not the endpoint of a rainbow $P_4$-copy. 

\textbf{Case 1:} $v$ lies in a rainbow $C_4$-copy, $C$.

We'll label the vertices of $C$, so that $C = v x y z v$, and the edge colors on $C$ are from $\{1,2,3,4\}$. Since $d(v) \geq 6$, $v$ is incident to at least two edges which are not colored from $\{1,2,3,4\}$. One of these edges may be incident to $y$, the vertex on $C$ to which we have not already specified an adjacency, but one must be incident to a new vertex, say $w$. We draw the situation in Figure~\ref{fig9}.

\begin{figure}[h]
\centering
\begin{tikzpicture}
\filldraw (0,0) circle(0.05 cm) node[below left]{$v$};
\filldraw (-1,1) circle(0.05 cm) node[left]{$x$};
\filldraw (0,2) circle(0.05 cm) node[above]{$y$};
\filldraw (1,1) circle(0.05 cm) node[right]{$z$};
\filldraw (0,-2) circle(0.05 cm) node[below left]{$w$};
\draw (0,0) -- (-1,1) node[pos=0.6, below]{1};
\draw (-1,1)-- (0,2) node[pos=0.4,above]{2};
\draw (0,2) -- (1,1) node[pos=0.6,above]{3};
\draw (1,1) -- (0,0) node[pos=0.4,below]{4};
\draw (0,0) -- (0,-2) node[pos=0.5,right]{5};

\end{tikzpicture}
\caption{}\label{fig9}
\end{figure}

Consider the possible neighbors of $w$. If $w$ is adjacent to $x$, then $c(wx)$ must be $3$, else $v w x y z$ is a rainbow $P_4$-copy ending at $v$. Similarly, if $wz$ is an edge, then $c(wz) = 2$. If $w$ is adjacent to a new vertex, say $u$, then we have $c(wu) \in \{2,3\}$, or else one of $u w v x y z$, $u w v z y x$ is a rainbow $P_5$-copy. Thus, if $w$ is adjacent to $u$, then it is not adjacent to one of $x,z$; if $w$ is adjacent to two new vertices, then it is not adjacent to either $x$ or $z$. Finally, $w$ may be adjacent to $y$. We conclude from this analysis that $d(w) \leq 4$.

The vertex $w$ is one of the neighbors of $v$ which we will ultimately include in $L(v)$, so we must verify that $w$ has no other neighbor of degree greater than $5$. Observe first $d(x) \leq 5$, since any edge incident to $x$ whose other endpoint is not on $C$ must be colored from $\{1,2,3,4\}$ to avoid creating a rainbow $P_4$-copy ending at $v$. Analogously, $d(z) \leq 5$. 

Suppose next that $w$ is adjacent to a vertex $u$ which is not on $C$. We have seen that $c(uw) \in \{ 2,3 \}$; since the colors are to this point symmetric, we may assume that $c(uw) = 2$. We now bound $d(u)$. 

First, if $u$ is adjacent to a vertex not yet described, say $s$, then $c(us) \in \{ 3,4,5 \}$, since otherwise $s u w v z y$ is a rainbow $P_5$-copy. Note that $u$ is not adjacent to $z$, since if so, either $v x y z u$ or $v w u z y$ is a rainbow $P_4$-copy ending at $v$. If $ux$ is an edge, then $c(ux) \in \{3,4\}$, since otherwise $v z y x u$ is a rainbow $P_4$-copy ending at $v$.
Thus, $d(u) \leq 5$, since $u$ is adjacent to $w$, may be adjacent to $y$, and any other incident edge to $u$ must be colored from $\{3,4,5\}$. 

Finally, $w$ may be adjacent to $y$. We wish to show that, if $wy$ is an edge, then $d(y) \leq 5$. Suppose for a contradiction that $wy$ is an edge and $d(y) \geq 6$. By Corollary~\ref{distance-2}, $y \in \overline{V'}$, and so by Lemma~\ref{length-4}, $y$ is not the endpoint of a rainbow $P_4$-copy. 

We re-draw the situation in Figure~\ref{fig10}, setting $c(yw) = c$.

\begin{figure}[h]
\centering
\begin{tikzpicture}
\filldraw (0,0) circle(0.05 cm) node[left]{$x$};
\filldraw (0,-2) circle(0.05 cm) node[below]{$v$};
\filldraw (0,2) circle(0.05 cm) node[above]{$y$};
\filldraw (-2,0) circle(0.05 cm) node[left]{$w$};
\filldraw (2,0) circle(0.05 cm) node[right]{$z$};

\draw (0,0) -- (0,2) node[pos=0.5, right]{2};
\draw (0,0) -- (0,-2) node[pos=0.5, right]{1};
\draw (0,2) -- (2,0) node[pos=0.5, above right]{3};
\draw (2,0) -- (0,-2) node[pos=0.5, below right]{4};
\draw (0,-2) -- (-2,0) node[pos=0.5, below left]{5};
\draw (-2,0) -- (0,2) node[pos=0.5, above left]{$c$};

\end{tikzpicture}
\caption{}\label{fig10}
\end{figure}

Now, consider $w$. If $w$ is adjacent to a vertex not depicted in Figure~\ref{fig10}, say $s$, then one of $y x v w s$, $y z v w s$ is a rainbow $P_4$-copy ending at $y$. So $w$ is adjacent to no vertex which is not drawn in Figure~\ref{fig10}. If $wx$ is an edge, then either $v w x y z$ is a rainbow $P_4$-copy ending at $v$, or $y x w v z$ is a rainbow $P_4$-copy ending at $y$. Finally, if $wz$ is an edge, then either $y z w v x$ is a rainbow $P_4$-copy ending at $y$ or $v w z y x$ is a rainbow $P_4$-copy ending at $v$. We conclude that $d(w) = 2$, a contradiction. So, if $wy$ is an edge, then $d(y) \leq 5$. 

We have thus shown that $d(w) \leq 4$ and $v$ is the only vertex in $N(w)$ of degree greater than $5$. We are now ready to build $L(v)$.

Observe that $v$ has at least $d(v) - 5$ neighbors of the same type as $w$, namely, neighbors which do not lie on $C$ and are incident to $v$ by an edge whose color is not from $\{1,2,3,4\}$. Let $L(v)$ be the set of such vertices. The above argument then shows that the maximum degree in $L(v)$ is at most $4$, and that if $w$ is any vertex in $L(v)$, then $v$ is the only neighbor of $w$ with degree greater than $5$. 

We now observe that the average degree in $L(v) \cup \{v\}$ is at most $5$. Say $|L(v)| = d(v) - k$; we have noted that $k \leq 5$. The average degree in $L(v) \cup \{v\}$ is then
$$\frac{d(v) + \underset{u \in L(v)}{\sum} d(u)}{|L(v)| + 1} \leq \frac{d(v) + 4(d(v) - k)}{d(v) - k + 1} = \frac{5(d(v) - 4k/5)}{d(v) - (k - 1)}.$$
Observe that 
$$\frac{5(d(v) - 4k/5)}{d(v) - (k - 1)} \leq 5$$
as long as $d(v) - (k-1) \geq d(v) - 4k/5$, i.e., $k - 1 \leq 4k/5$, which holds precisely when $k \leq 5$. 

So, $\{v\} \cup L(v)$ is a local pairing as desired.

\textbf{Case 2:} $v$ does not lie in a rainbow $C_4$-copy.

Using Lemma \ref{P3 lemma}, we know that $v$ is the endpoint of a rainbow $P_3$-copy, say $P = v x y z$, with edge colors $1,2,3$. Moreover, since $d(v) \geq 6$, $v$ is incident to at least three edges which are not colored from $\{1,2,3\}$. Clearly, one of these does not have its other endpoint on $P$. So there exists a new vertex $w$ such that $vw$ is an edge which receives a new color, say $4$. 

Consider the possible neighbors of $w$. If $w$ is adjacent to a vertex not on $P$, say $u$, then $c(wu) \in \{1,2,3\}$ in order to obey coloring rules and ensure that $u w v x y z$ is not a rainbow $P_5$-copy. If $wx$ is an edge, then $v w x y z$ is a $P_4$-copy ending at $v$, so must not be rainbow, which means $c(wx) = 3$. If $wy$ is an edge, then $w y x v w$ is a $C_4$-copy containing $v$, so is not rainbow, meaning $c(wy) = 1$. If $wz$ is an edge, then $v w z y x$ is a $P_4$-copy ending at $v$, so cannot be rainbow, meaning $c(wz) = 2$. Thus, $w$ can only be incident to edges colored from $\{1,2,3,4\}$, so we conclude that $d(w) \leq 4$.

Our goal is to use $w$ as one of the low-degree vertices in $L(v)$, so we must verify that $w$ has no neighbor other than $v$ of degree greater than $6$. 

Firstly, we examine $z$. If $z$ is adjacent to a vertex not in $\{w,v,x,y\}$, say $u$, then $c(zu) \in \{1,2\}$, else $v x y z u$ is a rainbow $P_4$-copy ending at $v$. We have seen that if $zw$ is an edge, then $c(zw) = 2$. If $vz$ is an edge, then $c(vz) = 2$, or else $v x y z v$ is a rainbow $C_4$-copy containing $v$. So, there is at most one vertex, namely $x$, which is incident to $z$ via an edge which is not colored from $\{1,2,3\}$. Thus, $d(z) \leq 4$.

Next, consider $x$. Suppose $wx$ is an edge and (for a contradiction) that $d(x) \geq 6$. Since $x$ is adjacent to $v$, we know by Lemma \ref{neighbor-lemma} that $x \in \overline{V'}$. Thus, by Lemma \ref{length-4}, $x$ is not the endpoint of a rainbow $P_4$-copy. 

We have seen that $c(wx)$ must be $3$. Observe that, since $d(v) \geq 6$, $v$ must be incident to an edge of a new color, say $5$, whose other endpoint is a vertex not yet considered, say $u$. We draw this in Figure~\ref{fig13}.

\begin{figure}[h]
\centering
\begin{tikzpicture}
\filldraw (0,0) circle(0.05 cm) node[below]{$w$};
\filldraw (2,0) circle(0.05 cm) node[below]{$v$};
\filldraw (4,0) circle(0.05 cm) node[below]{$x$};
\filldraw (6,0) circle(0.05 cm) node[below]{$y$};
\filldraw (8,0) circle(0.05 cm) node[below]{$z$};
\draw (0,0) -- (8,0);
\draw (1, 0) node[above]{4};
\draw (3, 0) node[above]{1};
\draw (5, 0) node[above]{2};
\draw (7, 0) node[above]{3};

\draw (0,0) arc(180:0:2);
\draw (2,2) node[above]{3};

\filldraw (0,-2) circle(0.05 cm) node[below]{$u$};
\draw (2,0) -- (0,-2) node[pos = 0.5, below right]{5};

\end{tikzpicture}
\caption{}\label{fig13}
\end{figure}

We shall achieve a contradiction by showing that $d(u) \leq 2$. First, we argue that $u$ is not adjacent to any vertex from $\{w,x,y,z\}$. Observe that $uw$ is not an edge, since if so, either $v u w x y$ is a rainbow $P_4$-copy ending at $v$, or $v u w x v$ is a rainbow-$C_4$ copy containing $v$. Next, $ux$ is not an edge, since if so, any legal choice for $c(ux)$ makes $v u x y z$ a rainbow $P_4$-copy ending at $v$. Observe that $uy$ is not an edge, since if so, either $v u y x w v$ is a rainbow $C_5$-copy containing $v$, or $v u y x v$ is a rainbow $C_4$-copy containing $v$. Finally, $uz$ is not an edge, since if so, either $v u z y x$ is a rainbow $P_4$-copy ending at $v$ or $x v u z y$ is a rainbow $P_4$-copy ending at $x$. 

Thus, $u$ is not adjacent to $w,x,y,$ or $z$. Suppose now that $u$ has a neighbor $s$ which is not yet considered. Then $s u v w x$ is a $P_4$-copy ending in $x$, so cannot be rainbow, which means $c(us) \in \{3,4\}$. We also have that $s u v x y z$ is a $P_5$-copy, so cannot be rainbow, which means $c(us) \neq 4$. Thus, if $u$ is adjacent to a new vertex, the edge used must be colored $3$, meaning that $u$ has at most one neighbor not yet drawn. So $d(u) \leq 2$, a contradiction. We conclude that if $wx$ is an edge, then $d(x) \leq 5$.

Next, suppose $wy$ is an edge. We wish to show that $d(y) \leq 5$. Suppose for a contradiction that $d(y) \geq 6$. We have seen that $c(wy) = 1$. Note that by Corollary \ref{distance-2}, $y$ is in $\overline{V'}$, so is not the endpoint of a rainbow $P_4$-copy.

We now examine $z$. Note first that if $z$ has a neighbor not in $\{w, v, x\}$, say $u$, then $c(zu) = 1$, or else one of $v x y z u$, $v w y z u$ is a rainbow $P_4$-copy ending at $v$. Thus, to achieve degree $3$, $z$ must have at least one neighbor among $\{w, v, x\}$. We shall show that, to the contrary, $z$ cannot be adjacent to any of these vertices. 

Suppose $wz$ is an edge. We have seen that $c(wz)$ must be $2$. But now $y z w v x$ is a rainbow $P_4$-copy ending in $y$, a contradiction. So $wz$ is not an edge. Observe that $zv$ is not an edge, since for any legal choice of $c(zv)$,  $v z y w v$ is a rainbow $C_4$-copy containing $v$. Finally, suppose $xz$ is an edge. We must have $c(xz) = 4$, else $v w y z x$ is a rainbow $P_4$-copy ending at $v$. Now, since $d(v) = 6$, $v$ must be incident to an edge of a new color, say $5$, whose other endpoint is not in $\{w,x,y,z\}$. Say this edge is $vs$. Observe that $yzxvu$ is now a rainbow $P_4$-copy ending at $y$. We conclude that $z$ is adjacent to none of $v, y, w$, so $d(z) \leq 2$, a contradiction. Thus, if $wy$ is an edge, then $d(y) \leq 5$. 

Finally, we must show that if $w$ is adjacent to a vertex not in $\{v,x,y,z\}$, say $u$, then $d(u) \leq 5$. Suppose to the contrary that $d(u) \geq 6$. We can again apply Corollary \ref{distance-2} to give that $u$ is in $\overline{V'}$, so $u$ is not the endpoint of a rainbow $P_4$-copy. We must have $c(uw) \in \{1,2\}$, else $u w v x y$ is a rainbow $P_4$-copy ending at $u$; however, unlike the other cases, we cannot at the moment fix $c(uw)$. We shall set $c(uw) = a$. 

Now, since $d(u) \geq 6$, $u$ is incident to an edge of a new color, $5$. It is simple to check that the other endpoint of this edge must be a vertex not in $\{v,x,y,z\}$, say $s$. Similarly, $v$ must be incident to an edge of another new color, $6$, and the endpoint of this edge must be a vertex not in $\{s,u,w,x,y,z\}$, say $t$. We illustrate the situation in Figure \ref{fig14}

\begin{figure}[h]
\centering
\begin{tikzpicture}
\filldraw (0,0) circle(0.05 cm) node[below left]{$w$};
\filldraw (2,0) circle(0.05 cm) node[below left]{$v$};
\filldraw (4,0) circle(0.05 cm) node[below]{$x$};
\filldraw (6,0) circle(0.05 cm) node[below]{$y$};
\filldraw (8,0) circle(0.05 cm) node[below]{$z$};
\draw (0,0) -- (8,0);
\draw (1, 0) node[above]{4};
\draw (3, 0) node[above]{1};
\draw (5, 0) node[above]{2};
\draw (7, 0) node[above]{3};

\filldraw (0,-2) circle(0.05 cm) node[below left]{$u$};
\draw (0,0) -- (0,-2) node[pos = 0.5, left]{$a$};

\filldraw (0,-4) circle(0.05 cm) node[below]{$s$};

\draw (0,-2) -- (0,0-4) node[pos = 0.5, left]{$5$};

\filldraw (2,-2) circle(0.05 cm) node[below]{$t$};

\draw (2,0) -- (2,-2) node[pos=0.5, right]{6};

\end{tikzpicture}
\caption{}\label{fig14}
\end{figure}

We will argue that $d(t) \leq 2$. We claim  first that $t$ is not adjacent to $s, u, w,$ or $x$. Indeed, if $ts$ is an edge, then one of the paths $ustvw, ustvx$ is a rainbow $P_4$-copy ending at $u$. If $tu$ is an edge, then either $v$ is in a rainbow $C_4$-copy or  $utvxy$ is a rainbow $P_4$-copy ending at $u$. If $tw$ is an edge, then either $vtwus$ is a rainbow $P_4$-copy ending at $v$, or $wtvxyz$ is a rainbow $P_5$-copy. And if $tx$ is an edge, then either $vtxyz$ is a rainbow $P_4$-copy ending at $v$, or $uwvtx$ is a rainbow $P_4$-copy ending at $u$. Thus, the only possible neighbors of $t$ are $y,z$, and vertices not depicted in Figure~\ref{fig14}.

Observe, if $ty$ is an edge, then $c(ty) \in \{a, 4\}$, else $uwvty$ is a rainbow $P_4$-copy ending at $u$. But $c(ty) \neq 4$, since then $vtyxv$ is a rainbow $C_4$-copy containing $v$. Thus, if $ty$ is an edge, then $c(ty) = a$. Similarly, if $tz$ is an edge, then $c(tz) = a$. And if $t$ is adjacent to a vertex not yet drawn, then the incident edge also must be colored $a$ to avoid forming either a rainbow $P_5$-copy with $tvxyz$ or a rainbow $P_4$-copy ending at $u$. 

Thus, if $t$ is adjacent to any vertex other than $v$, the incident edge must be colored $a$, so $d(t) \leq 2$, a contradiction.

We have thus argued that $d(w) \leq 4$ and that $v$ is the only vertex in $N(w)$ of degree greater than $5$. We are now ready to build $L(v)$.

Note that $v$ has at least $d(v) - 4$ neighbors of the same type as $w$, that is, neighbors which are not on the path $vxyz$ and which are incident to $v$ by an edge whose color is not from $\{1,2,3\}$. (Note that if $vz$ is an edge, then $c(vz) = 2$, else $v$ is in a rainbow-$C_4$ copy, and thus we can find $d(v) - 4$ vertices of the same type as $w$, instead of the expected $d(v) - 5$.) Let $L(v)$ be the set of such vertices. As in the previous case, it follows that the average degree in $L(v) \cup \{v\}$ is at most $5$, (and actually will be strictly less than $5$, as $|L(v)| > d(v) - 5$). 
\end{proof}

\section{Concluding remarks}

While the details of the proofs are largely routine case analysis, it is worth reiterating the central idea which makes this approach to our problem tractable. The choice to partition $V$ as $V' \cup \overline{V'}$ is not arbitrary. Since rainbow $C_{\ell}$-copies arise naturally when we attempt to avoid rainbow $P_{\ell}$-copies, it is much easier to show that rainbow $P_{\ell}$-copies are unavoidable if we start with a vertex in $\overline{V'}$ than with a vertex whose situation we do not know. On the other hand, every vertex in $V'$ lies on a rainbow $C_5$-copy, meaning that we can begin any case analysis concerning vertices in $V'$ with a rainbow $C_5$-copy, rather than with a single vertex. This additional guaranteed structure likewise makes case analysis arguments tractable where they seem not to be if we pick a vertex from $V$ without assuming that it lies (or does not lie) in a rainbow $C_5$-copy. 

Given the increase of required analysis from previous results (the proof that $\ex^*(n, P_4) = 2n + O(1)$ is about two pages long), it seems clear that these naive methods will not be sufficient to obtain exact results for longer paths. However, the result for $P_5$ alone adds weight to the conjecture that $\ex^*(n,P_{\ell}) = \frac{\ell }{2}n + O(1)$ in general.

\section*{Acknowledgements}

The author would like to thank Cory Palmer for his advice and support in the preparation of this paper.

\end{document}